\documentclass[a4paper,11pt, twoside]{article}

\usepackage[figuresright]{rotating}
\usepackage{graphicx}          % Include this line if your 
\usepackage[latin1]{inputenc}
\usepackage[english]{babel}
\usepackage[T1]{fontenc}
\usepackage{amsmath, amssymb}
\usepackage{amsthm}
\usepackage{color}
\usepackage{setspace}
\usepackage{pdfcprot}
\usepackage{paralist}
\usepackage{tikz}
\usepackage{pdfcprot}
\usepackage{pgfplots}

\pgfplotscreateplotcyclelist{error}{%
{color=blue,dashed,mark=*, mark repeat=50},
{color=black,dashed,mark=triangle*, mark repeat=50},
{color=red,densely dashed,line width=1.5pt},
{color=blue, dotted,solid},
{color=black, dashdotted,solid},
}
\pgfplotscreateplotcyclelist{error2}{%
{color=blue,dashed,mark=*},
{color=black,dashed,mark=triangle*},
{color=red,densely dashed,line width=1.5pt},
{color=blue, dotted,mark=o,mark options={solid}},
{color=black, dashdotted,mark=triangle,mark options={solid}},
}
\newcommand{\diag}{\operatorname{diag}}

\newcommand{\vect}{\operatorname{vec}}

\newcommand{\trace}{\operatorname{tr}}

\newcommand{\expn}{\operatorname{e}}

\newcommand{\R}{\mathbb{R}}

\newcommand{\matlab}{MATLAB\textsuperscript{\textregistered}  }
 \newcommand{\intel}{Intel\textsuperscript{\textregistered}}
\newcommand{\xeon}{Xeon\textsuperscript{\textregistered}}

\newtheorem{defn}{Definition}[section]
\newtheorem*{bem}{Remark}

\newtheorem{lem}[defn]{Lemma}

\newtheorem{thm}[defn]{Theorem}

\begin{document}

\title{An $\mathcal H_2$-Type Error Bound for Time-Limited Balanced Truncation}

\author{Martin Redmann\thanks{Corresponding author. Weierstrass Institute for Applied Analysis and Stochastics, Mohrenstr. 39, 10117 Berlin, Germany, 
Email: {\tt martin.redmann@wias-berlin.de}. Financial support by the DFG via Research Unit FOR 2402 is gratefully acknowledged.} \and Patrick 
K{\"u}rschner\thanks{Max Planck Institute for Dynamics of Complex Technical Systems, 
Sandtorstr.\ 1, 39106 Magdeburg, Germany, Email: {\tt kuerschner@mpi-magdeburg.mpg.de}}}

\maketitle

\begin{abstract}                         
When solving partial differential equations numerically, usually a high order spatial discretization is needed.
Model order reduction (MOR) techniques are often used to reduce the order of spatially-discretized systems and hence reduce 
computational complexity. A particular MOR technique to obtain a reduced order model (ROM) is balanced truncation (BT). However, if one aims at 
finding a good ROM on a certain finite time interval only, time-limited BT (TLBT) can be a more accurate alternative. So far, no error bound on TLBT 
has been proved. In this paper, we close this gap in the theory by providing an $\mathcal H_2$ error bound for TLBT with two different 
representations. The performance of the error bound is then shown in several numerical experiments.
\end{abstract}


\setlength{\parindent}{0pt} %Einschub nach vorne gesetzt
   \section{Introduction}
  Let $(A,B,C)\in\mathbb{R}^{n\times n}\times\mathbb{R}^{n\times m}\times\mathbb{R}^{p\times m}$ be a realization of a linear, time-invariant system
\begin{align}\label{lti}
\boldsymbol{\Sigma}:\quad \dot x(t)=Ax(t)+Bu(t),\quad x(0)=0,\quad y(t)=Cx(t)
\end{align}
and assume that $A$ is Hurwitz which implies \eqref{lti} is asymptotically stable. 
The infinite reachability and observability Gramians
\begin{align*}
	P_{\infty}=\int_0^{\infty} \expn^{As}BB^T \expn^{A^Ts} ds,\quad  Q_{\infty}=\int_0^{\infty} \expn^{A^Ts}C^TC \expn^{As} ds
\end{align*}
of $(A,B,C)$ solve the Lyapunov equations
\begin{align}\label{BT_Lyap}
AP_{\infty}+P_{\infty}A^T+BB^T=0,\quad A^TQ_{\infty}+Q_{\infty}A^T+C^TC=0.
\end{align}
The first ingredient of balanced truncation~\cite{moore} (BT) is to
simultaneously diagonalize both Gramians through congruence transformations $\hat S P_{\infty} \hat S^T=\hat S^{-T}Q_{\infty}\hat 
S^{-1}=\Sigma_{\infty}$  which gives a balanced realization $(\hat S A \hat S^{-1}, \hat SB, C\hat S^{-1})$, where $\Sigma_{\infty}$ is diagonal and 
contains the Hankel singular values $\sigma_j$ (HSVs),
i.e., the square root of the eigenvalues of $P_{\infty}Q_{\infty}$. In the second step the reduced order model $\boldsymbol{\Sigma}_r$ is obtained by 
keeping only the $r\times r$ upper left block of $\hat S A \hat S^{-1}$ and the associated parts 
of $\hat SB, C\hat S^{-1}$, i.e., the smallest $n-r$ HSVs are removed from the system. With Cholesky factorizations $P_{\infty}=L_PL_P^T$, 
$Q_{\infty}=L_QL_Q^T$, and the singular value decomposition (SVD) 
$X\Sigma_{\infty}Y^T=L_Q^TL_P$, the balancing 
transformation is given by $\hat S=L_QX\Sigma_{\infty}^{-\tfrac{1}{2}}$ and $\hat S^{-1}=L_PY\Sigma_{\infty}^{-\tfrac{1}{2}}$, see, 
e.g.,~\cite{antoulas}. This leads to non increasingly ordered $\sigma_j$. Moreover, the resulting reduced system $\boldsymbol{\Sigma}_r$ is 
asymptotically stable and satisfies the $\mathcal H_{\infty}$ error bound~\cite{morGlo84}
 \begin{align}\label{BT_errorB}
  \|\boldsymbol{\Sigma}-\boldsymbol{\Sigma}_r\|_{\mathcal H_{\infty}}\leq2(\sigma_{r+1}+\ldots+\sigma_n).
 \end{align}
Once the SVD is computed,~\eqref{BT_errorB} can be used to adaptively adjust the reduced order $r$. A generalized $\mathcal 
H_{\infty}$-error bound for BT has been proved in \cite{bennerdammcruz, dammbennernewansatz}, where linear stochastic system are investigated.\\
The matrix of truncated HSVs 
$\Sigma_2=\diag(\sigma_{r+1}, \ldots, \sigma_n)$ can be used to express the $\mathcal{H}_2$ error bound \cite{antoulas}. It is represented by 
\begin{align}\label{h2_bound_infty}
  \|\boldsymbol{\Sigma}-\boldsymbol{\Sigma}_r\|^2_{\mathcal H_{2}}\leq\trace(\Sigma_{2} (B_2 B_2^T+2 P_{\infty, M, 2} A_{21}^T)),\end{align}
where $B_2$ is the matrix of the last $n-r$ rows of $\hat S B$, $A_{21}$ is the left lower $(n-r)\times r$ block of $\hat S A \hat S^{-1}$ and 
$P_{\infty, M, 2}$ are the last $n-r$ rows of the mixed Gramian $P_{\infty, M}=\hat S \int_0^{\infty} \expn^{As}BB_1^T \expn^{A_{11}^Ts} ds$. 
The bound in (\ref{h2_bound_infty}) has already been extended to stochastic systems in a more general form \cite{redmannbenner, 
redmannfreitag, BTtyp2EB}.

In~\cite{morGawJ90} Gawronski and Juang restricted balanced truncation to a finite time interval $[0,\bar T]$, $\bar T<\infty$, by
introducing the time-limited reachability and observability Gramians
\begin{align}\label{TLBT_Gram}
	P_{\bar T}:=\int_0^{\bar T} \expn^{As}BB^T \expn^{A^Ts} ds,\quad  Q_{\bar T}=\int_0^{\bar T} \expn^{A^Ts}C^TC \expn^{As} ds.
\end{align}
It is easy to show that $P_{\bar T},~Q_{\bar T}$ solve the Lyapunov equations
\begin{align}\label{TLBT_Lyap}
AP_{\bar T}+P_{\bar T}A^T+BB^T-F_{\bar T}F_{\bar T}^T&=0,\\
A^TQ_{\bar T}+Q_{\bar T}A^T+C^TC-G_{\bar T}^TG_{\bar T}&=0,
\end{align}
where $G_t:=C \expn^{At}$ and $F_t:=\expn^{At}B$, $t\in[0, {\bar T}]$. Time-limited balanced truncation (TLBT) is then carried out by
using the Cholesky factors of $P_{\bar T}$, $Q_{\bar T}$ instead of $P_{\infty}, Q_{\infty}$ to construct the balancing transformation 
which in this case is denoted by $S$. 
This transformation simultaneously diagonalizes $P_{\bar T}$, $Q_{\bar T}$, i.e., $S P_{\bar T} S^T=S^{-T}Q_{\bar 
T}S^{-1}=\Sigma_{\bar T}$ and is, thus, referred to as time-limited balancing transformation. The values in $\Sigma_{\bar T}$ are referred to as 
time-limited singular values and are, similar to the HSVs, invariant under state-space transformations. Because of the altered Gramian definitions, 
TLBT does generally not preserve stability and there is no $\mathcal H_{\infty}$ error bound as in unrestricted BT.\smallskip

The main contribution of this paper is a generalized $\mathcal H_2$ error bound for TLBT. It leads to (\ref{h2_bound_infty}) if $\bar 
T\rightarrow \infty$. We provide two representations of this bound. The first one can be used for practical computations and is, hence, an important 
tool to assess the obtained accuracy. The second representation is not appropriate for computing the bound but it shows
that, similar to BT, the 
time-limited singular values deliver an alternative criterion to find a suitable reduced order dimension $r$. We conclude this paper
by conducting several numerical 
experiments which indicate that the time-limited $\mathcal H_2$ bound is tight.
%%%%%%%%%%%%%%%%%%%%%%%%%%%%%%%%%%%%%%%
 \section{$\mathcal H_2$-type Error Bounds for Time-Limited Balanced Truncation}\label{sec2}
Let $S$ be the time-limited balancing transformation. We partition the balanced realization $(S A 
S^{-1}, 
SB, CS^{-1})$ as follows: \begin{align*}
S A S^{-1}= \begin{bmatrix}{A}_{11}&{A}_{12}\\ 
{A}_{21}&{A}_{22}\end{bmatrix},\;\;\; S B = \begin{bmatrix}{B}_1\\ {B}_2\end{bmatrix},\;\;\;  
C S^{-1} = \begin{bmatrix}{C}_1 &{C}_2\end{bmatrix}, \end{align*}
where ${A}_{11}\in\R^{r\times r}$, $B_{1}\in\R^{r\times m}$, $C_{1}\in\R^{p\times r}$ and the other blocks of appropriate dimensions. Furthermore, we 
introduce    
{\allowdisplaybreaks \begin{align*}
S F_{\bar T}=\begin{bmatrix}
F_{{\bar T}, 1}\\
F_{{\bar T}, 2}\end{bmatrix},\;G_{\bar T} S^{-1}=\begin{bmatrix}
 G_{{\bar T}, 1}& G_{{\bar T}, 2}\end{bmatrix},\;\Sigma_{\bar T}=\begin{bmatrix}
 \Sigma_{{\bar T}, 1}& \\
 & \Sigma_{{\bar T}, 2}
\end{bmatrix}.\end{align*}}
 We consider the corresponding Lyapunov equations in partitioned form:
{\allowdisplaybreaks \begin{align} 
\label{donotusenewgram}
\left[\begin{smallmatrix}{A}_{11}&{A}_{12}\\ 
{A}_{21}&{A}_{22}\end{smallmatrix}\right] \left[\begin{smallmatrix}\Sigma_{{\bar T}, 1}&\\ 
&\Sigma_{{\bar T}, 2}\end{smallmatrix}\right]+\left[\begin{smallmatrix}\Sigma_{{\bar T}, 1} &\\ 
&\Sigma_{{\bar T}, 2}\end{smallmatrix}\right] \left[\begin{smallmatrix} A^T_{11}&A^T_{21}\\ 
A^T_{12}& A^T_{22} \end{smallmatrix}\right]=& -\left[\begin{smallmatrix} {B}_1 B_1^T& {B}_1 B_2^T \\ 
{B}_2 B_1^T& {B}_2 B_2^T \end{smallmatrix}\right]\\\nonumber
&+\left[\begin{smallmatrix} F_{\bar T,1} F_{\bar T,1}^T& F_{\bar T,1} F_{\bar T,2}^T \\ 
F_{\bar T,2}F_{\bar T,1}^T& F_{\bar T,2}F_{\bar T,2}^T \end{smallmatrix}\right], \\ \label{ersteeqbt2ob}
\left[\begin{smallmatrix}{A}^T_{11}&{A}^T_{21}\\ 
{A}^T_{12}&{A}^T_{22}\end{smallmatrix}\right]\left[\begin{smallmatrix}\Sigma_{{\bar T}, 1}&\\ 
&\Sigma_{{\bar T}, 2}\end{smallmatrix}\right]+\left[\begin{smallmatrix}\Sigma_{{\bar T}, 1}&\\ 
&\Sigma_{{\bar T}, 2}\end{smallmatrix}\right]\left[\begin{smallmatrix}A_{11}&A_{12}\\ 
A_{21}& A_{22}\end{smallmatrix}\right]=& -\left[\begin{smallmatrix}C_1^TC_1&C_1^TC_2 \\ 
C^T_2C_1& C^T_2C_2 \end{smallmatrix}\right]\\\nonumber
&+\left[\begin{smallmatrix} G_{\bar T,1}^T G_{\bar T,1}& G_{\bar T,1}^TG_{\bar T,2} \\ 
G_{\bar T,2}^TG_{\bar T,1}& G_{\bar T,2}^TG_{\bar T,2} \end{smallmatrix}\right].
\end{align}}
The TLBT reduced system that approximates (\ref{lti}) is given by
\begin{align*}
\dot x_r(t)=A_{11}x_r(t)+B_1u(t),\quad x_r(0)=0,\quad y_r(t)=C_1 x_r(t).
\end{align*}
The goal of this section is to find a bound for the error between $y$ and $y_r$. Since we have zero initial conditions for both the reduced and the 
full system, we have the following representations for the outputs {\allowdisplaybreaks \begin{align*}
 y(t)&=C x(t)=C \int_0^t \expn^{A(t-s)} B u(s) ds, \\
y_r(t)&=C_1 x_r(t)=C_1 \int_0^t  \expn^{A_{11}(t-s)} B_1 u(s) ds,
\end{align*}}
where $t\in [0, {\bar T}]$. To find a first representation for the error bound, arguments from \cite{redmannbenner, 
redmannfreitag, BTtyp2EB} are used. There a generalized $\mathcal H_2$ error bound for stochastic systems has been derived. Some easy rearrangements 
yield a first error estimate  {\allowdisplaybreaks \begin{align*}
&\left\|y(t)- y_r(t)\right\|_2 \\&=\left\|C \int_0^t \expn^{A(t-s)} B u(s) ds- C_1 \int_0^t  \expn^{A_{11}(t-s)} B_1 u(s) ds\right\|_2\\&\leq  
\int_0^t \left\|\left(C \expn^{A(t-s)}  B - C_1 \expn^{A_{11}(t-s)} B_1\right) u(s)\right\|_2 ds\\&\leq \int_0^t \left\| C 
\expn^{A(t-s)} B - C_1 \expn^{A_{11}(t-s)} B_1\right\|_F \left\|u(s)\right\|_2 ds.
\end{align*}}
By the Cauchy Schwarz inequality it holds that{\allowdisplaybreaks \begin{align*}
&\left\|y(t)- y_r(t)\right\|_2\\&\leq 
\left(\int_0^t \left\|C \expn^{A(t-s)} B - C_1 \expn^{A_{11}(t-s)} B_1\right\|_F^2 ds\right)^{\frac{1}{2}} \left(\int_0^t 
\left\|u(s)\right\|_2^2 ds\right)^{\frac{1}{2}}.
\end{align*}}
Using substitution, the definition of the Frobenius norm and the linearity of the integral, we obtain 
{\allowdisplaybreaks \begin{align*}
 & \int_0^t \left\| C \expn^{A(t-s)} B - C_1 \expn^{A_{11}(t-s)} B_1\right\|_F^2 ds\\&=\int_0^t \left\|C \expn^{As} B - C_1 \expn^{A_{11}s} 
B_1\right\|_F^2 ds \\&\leq \int_0^{\bar T} \left\|C \expn^{As} B - C_1 \expn^{A_{11}s} B_1\right\|_F^2 ds\\&=\int_0^{\bar T} \trace\left(C 
\expn^{As}BB^T 
\expn^{A^Ts} C^T\right)ds\\&\quad +\int_0^{\bar T}\trace\left(C_1 \expn^{A_{11}s}B_1B_1^T\expn^{A_{11}^Ts}  C_1^T\right)ds\\&\quad 
-2\int_0^{\bar T}\trace\left(C 
\expn^{As}BB_1^T \expn^{A_{11}^Ts} C_1^T\right) ds\\&=\trace\left(C P_{\bar T} C^T\right)+\trace\left(C_1 P_{{\bar T}, r} 
C_1^T\right)-2\;\trace\left(C P_{{\bar T}, M} C_1^T\right),
\end{align*}}
where $P_{\bar T}:=\int_0^{\bar T}  \expn^{As}BB^T \expn^{A^Ts}ds$, $P_{{\bar T}, r}:=\int_0^{\bar T} \expn^{A_{11}s}B_1B_1^T\expn^{A_{11}^Ts} ds$ 
and 
$P_{{\bar T},M}:=\int_0^{\bar T} \expn^{As}BB_1^T \expn^{A_{11}^Ts} ds$. 
Matrix-valued integrals of this form can under some conditions be expressed as unique solutions of matrix equations.
\begin{lem}\label{timeint_mateq}
 Let $A_1\in\mathbb{R}^{n\times n},~A_2\in\mathbb{R}^{r\times r}$ with $\Lambda(A_1)\cap-\Lambda(A_2)=\emptyset$ and
$B_1\in\mathbb{R}^{n\times m}$, $B_2\in\mathbb{R}^{r\times m}$. Then,
\begin{align*}
 X=\int_{0}^{\bar T}\expn^{A_1s}B_1B_2^T \expn^{A_2^Ts}ds
\end{align*}
solves the Sylvester equation
\begin{align*}
 A_{1} X+X A_{2}^T &=-B_1 B_2^T+\expn^{A_1{\bar T}}B_1B_2^T\expn^{A_2^T{\bar T}}.
\end{align*}
\end{lem}
\begin{proof}
 The integral is equivalent to
 {\allowdisplaybreaks \begin{align*}
  \vect X&=\int_{0}^{\bar T}\vect\expn^{A_1s}B_1B_2^T \expn^{A_2^Ts}ds\\
  &=\int_{0}^{\bar T}\expn^{A_2s}\otimes \expn^{A_1s}ds\vect B_1B_2^T\\
  &=\int_{0}^{\bar T}\expn^{\left(I_r\otimes A_1 +A_2\otimes I_n \right)s}ds \vect B_1B_2^T,
 \end{align*}}
where we used \cite[Theorem~10.9]{Hig08}. The matrix $\mathcal{A}:=I_r\otimes A_1 +A_2\otimes I_n$ is nonsingular and it holds that
\begin{align*}
 \vect X&=\mathcal{A}^{-1}\left(\expn^{\mathcal{A}{\bar T}}-I_{nr}\right)\vect B_1B_2^T\\
 \Leftrightarrow\quad \mathcal{A}\vect X&=\left(\expn^{\mathcal{A}{\bar T}}-I_{nr}\right)\vect B_1B_2^T
\end{align*}
and the claim follows after de-vectorization.
\end{proof}
\begin{bem}
The result of the above Lemma is also a consequence of the product rule. Setting $g_1(t):=\expn^{A_1 t}B_1$ and $g_2(t):=B_2^T \expn^{A_2^T t}$, it 
holds 
that \begin{align*}
&g_1(\bar T)g_2(\bar T)-g_1(0)g_2(0)=\int_{0}^{\bar T}g_1(s)dg_2(s)+\int_{0}^{\bar T}d g_1(s) g_2(s)\\&= \int_{0}^{\bar T}g_1(s)g_2(s)ds\;A_2^T 
+A_1 \int_{0}^{\bar T} g_1(s) g_2(s) ds,
\end{align*}
since $dg_2(s)=g_2(s) A_2^T ds$ and $dg_1(s)=A_1 g_1(s) ds$.

The time-limited Gramians~\eqref{TLBT_Gram} also exists for unstable systems. Therefore, it is, e.g. in~\cite[Section 7.6.5]{antoulas}, discussed to 
use TLBT to
reduce unstable systems. The above Lemma further reveals that in this situation and if 
$\Lambda(A)\cap-\Lambda(A)=\emptyset$, the time-limited Gramians can still be obtained by solving the time-limited
Lyapunov equations~\eqref{TLBT_Lyap} which is important from a numerical point of view.
In this work, however, we will not pursue the reduction of unstable systems further.
\end{bem}
From now on we assume that $\Lambda(A_{11})\cap-\Lambda(A_{11})=\emptyset$ and $\Lambda(A)\cap-\Lambda(A_{11})=\emptyset$, implying by
Lemma~\ref{timeint_mateq} that
the matrices $P_{{\bar T}, r}$ and $P_{{\bar T}, M}$ are the unique solutions of 
{\allowdisplaybreaks 
\begin{subequations}\label{lyapeq_btbound}
\begin{align}\label{nonbalregramp}
A_{11} P_{{\bar T}, r}+P_{{\bar T},r} A_{11}^T &=-B_1 B_1^T+F_{{\bar T}, r} F_{{\bar T}, r}^T,\\ 
\label{mixeqbt2blsa} 
A P_{{\bar T}, M}+ P_{{\bar T}, M} A_{11}^T &=-B B_1^T+F_{{\bar T}} F_{{\bar T}, r}^T,                               
\end{align}
\end{subequations}
}
where $F_{{\bar T}, r}:=\expn^{A_{11}T}B_1$. We have, thus, established the following result.
\begin{thm}\label{thm_basic}
Let $\Lambda(A_{11})\cap-\Lambda(A_{11})=\emptyset$ and $\Lambda(A)\cap-\Lambda(A_{11})=\emptyset$. Then the following error bound holds for the 
reduced system
$\boldsymbol{\Sigma}_r$ generated by TLBT 
 {\allowdisplaybreaks \begin{align}\label{implicdeptilp}
\begin{split}
&\max_{t\in [0, {\bar T}]}\left\|y(t)- y_r(t)\right\|_2\leq\epsilon\left\|u\right\|_{L^2_{\bar T}},\\ 
&\epsilon:=\left(\trace \left(C P_{\bar T} C^T\right)+\trace \left(C_1 P_{{\bar T}, r} 
C_1^T\right)-2\trace \left(C P_{{\bar T},M} C_1^T\right)\right)^{\frac{1}{2}}.
\end{split}
\end{align}}
\end{thm}
The representation (\ref{implicdeptilp}) of the error bound has the same structure as the one computed in the stochastic framework 
\cite{redmannbenner, redmannfreitag, BTtyp2EB} but it is clearly different since solutions of different matrix equations enter in the time-limited 
case. The bound in (\ref{implicdeptilp}) can be used for practical computations. It only requires to solve the matrix equations in 
(\ref{lyapeq_btbound}) since 
$P_{\bar T}$ is already known from the balancing procedure. The matrix equations (\ref{lyapeq_btbound}) are not expensive since 
$P_{{\bar T}, r}$ usually is a small matrix and $P_{{\bar T},M}$ only has a few columns.\smallskip

The next theorem provides an alternative representation of 
this bound. It can be expressed with the help of $\Sigma_{{\bar T}, 2}=\diag(\sigma_{{\bar T}, r+1}, \ldots, \sigma_{{\bar T}, n})$ which is the 
matrix of truncated time-limited singular values. In \cite{redmannbenner, redmannfreitag, BTtyp2EB} representations of 
generalized $\mathcal{H}_2$ error bounds have been shown using the truncated HSVs of the underlying stochastic system. However, the matrix equations 
(\ref{TLBT_Lyap}) and (\ref{lyapeq_btbound}) have a very different structure than the generalized equations for stochastic system. Therefore, we 
need to apply other techniques in order to obtain the result below. This result also shows essential differences in its structure compared to the 
stochastic case.
\begin{thm}\label{thmmaimn}
Using the coefficients of the balanced realization of the system, the error bound in (\ref{implicdeptilp}) can be expressed as 
follows:{\allowdisplaybreaks \begin{align*}
& \trace\left(C P_{\bar T} C^T+ C_1 P_{{\bar T}, r} C_1^T-2 C P_{{\bar T}, M} C_1^T\right)\\=&\trace(\Sigma_{{\bar T}, 2} (B_2 B_2^T+2 P_{{\bar T}, 
M, 
2} A_{21}^T))-2\trace(G^T_{{\bar T}, 1} 
G_{\bar T} P_{{\bar T}, M})\\&+\trace(G^T_{{\bar T}, 1} G_{{\bar T}, 1} P_{{\bar T}, r})+\trace(F_{{\bar T}, 1} F^T_{{\bar T}, 1} \Sigma_{{\bar T}, 
1})\\&-\trace((F_{{\bar T}, 
1}-F_{{\bar T}, r}) (F_{{\bar T}, 1}-F_{{\bar T}, r})^T \Sigma_{{\bar T}, 1}),\end{align*}}
where $P_{{\bar T}, M, 2}$ are the last $n-r$ rows of $S P_{{\bar T}, M}$ with $S$ being the balancing transformation.
\begin{proof}
By selecting the left and right upper block of (\ref{ersteeqbt2ob}), we 
have{\allowdisplaybreaks \begin{align}\label{firstc1c2bt2jk}
A_{11}^T \Sigma_{{\bar T}, 1}+\Sigma_{{\bar T}, 1} A_{11} &=-C_1^TC_1+G_{{\bar T}, 1}^T G_{{\bar T},1}\\ \label{c1c2eqbt2}
A_{21}^T \Sigma_{{\bar T}, 2}+\Sigma_{{\bar T}, 1} A_{12} &=-C_1^T C_2+G_{{\bar T}, 1}^TG_{{\bar T},2}.
                                       \end{align}}
We introduce the reduced order system observability Gramian $Q_{{\bar T}, r}:=\int_0^{\bar T} \expn^{A^T_{11}s}C_1^T C_1\expn^{A_{11}s} ds$ which 
satisfies
{\allowdisplaybreaks \begin{align}\label{redgramobsbt2fzu}
   A_{11}^T Q_{{\bar T}, r}+Q_{{\bar T}, r} A_{11}  =-C_1^T C_1+G_{{\bar T}, r}^T G_{{\bar T}, r}
\end{align}}
with $G_{{\bar T}, r}:=C_1 \expn^{A_{11}{\bar T}}$. 
% We define {\allowdisplaybreaks \begin{align*}
%   \epsilon:=\sqrt{\trace(C P_{\bar T} C^T)+\trace(C_1 P_{{\bar T}, r} C_1^T)-2\trace(C P_{{\bar T}, M} C_1^T)}.
%           \end{align*}}
We make use of the integral representations of $P_{\bar T}$ and $Q_{\bar T}$ and apply properties of the trace. Hence, we have 
\begin{align*}
&\trace(C P_{\bar T} C^T)=\int_0^{\bar T} \trace(C \expn^{As}BB^T \expn^{A^Ts} C^T)ds\\&=\int_0^{\bar T} \trace(B^T \expn^{A^Ts} C^TC 
\expn^{As}B)ds=\trace(B^T 
Q_{\bar T} B).
\end{align*}
Using the balancing transformation $S$ and the partition of $SB$, we obtain \begin{align*}
\trace(B^TQ_{\bar T} B)&=\trace(B^TS^T 
S^{-T}Q_{\bar T}S^{-1}SB)=\trace(B^TS^T \Sigma_{\bar T}SB)\\&=\trace(B_1^T\Sigma_{\bar T, 1}B_1)+\trace(B_2^T \Sigma_{\bar T, 2} B_2).                
   \end{align*}
The partition of $CS^{-1}$ and $S P_{{\bar T}, M}=\left[\begin{smallmatrix}
P_{{\bar T}, M, 1} \\
P_{{\bar T}, M, 2}
\end{smallmatrix}\right]$ yield 
\begin{align*}
\trace(C P_{\bar T,M}C_1^T)&=\trace(C S^{-1} S P_{\bar T,M} C_1^T)\\&=\trace(C_1 P_{\bar T,M, 1} C_1^T)+\trace(C_2 P_{\bar T,M, 2} 
C_1^T).
\end{align*}
For $\epsilon$ in \eqref{implicdeptilp} this leads to 
\begin{align}\label{insertforebbt2fh}
\epsilon^2=&\trace(B_1^T \Sigma_{{\bar T},1} B_1)+\trace(B_2^T \Sigma_{{\bar T}, 2} B_2)+ 
\trace(C_1 P_{{\bar T}, r} C_1^T)\\ \nonumber&-2\trace(C_1 P_{{\bar T}, M, 1} C_1^T)-2\trace(C_2 P_{{\bar T}, M, 2} C_1^T).
          \end{align}
We insert equation (\ref{c1c2eqbt2}) which yields 
{\allowdisplaybreaks \begin{align*}
\trace(C_2 P_{{\bar T}, M, 2} C_1^T)&=\trace(P_{{\bar T},M, 2}  C_1^T C_2 )\\&=-\trace(P_{{\bar T}, M, 2} (A_{21}^T 
\Sigma_{{\bar T}, 2}+\Sigma_{{\bar T}, 1} A_{12}))\\&\quad+\trace(P_{{\bar T},M, 2}G_{{\bar T}, 1}^TG_{{\bar T},2} )\\&=-\trace(\Sigma_{{\bar T}, 2} 
P_{{\bar T}, M, 2} 
A_{21}^T)-\trace(\Sigma_{{\bar T}, 1} 
A_{12} P_{{\bar T}, M, 2})\\&\quad +\trace(G_{{\bar T}, 1}^TG_{{\bar T},2} P_{{\bar T},M, 2}).
          \end{align*}}
We multiply (\ref{mixeqbt2blsa}) with $S$ from the left and evaluate the resulting upper block of the equation: \begin{align*}
-A_{12} P_{{\bar T}, M, 2}=A_{11} P_{{\bar T}, M, 1}+P_{{\bar T}, M, 1} A_{11}^T+B_1 B_1^T-F_{{\bar T}, 1}F_{{\bar T}, r}^T.\\               
   \end{align*}
Hence, we have {\allowdisplaybreaks \begin{align*}
&-2\trace(C_2 P_{{\bar T}, M, 2} C_1^T)=\\&2[\trace( \Sigma_{{\bar T}, 1}F_{{\bar T}, 1}F_{{\bar T}, r}^T)-\trace( \Sigma_{{\bar T}, 1} (B_1 
B_1^T+A_{11} P_{{\bar T}, 
M, 1}+ P_{{\bar T}, M, 1} 
A_{11}^T))]\\&+2[\trace(\Sigma_2 P_{{\bar T}, M, 2} A_{21}^T)-\trace(G_{{\bar T}, 1}^TG_{{\bar T},2} P_{{\bar T},M, 2})].
          \end{align*}}
Using equation (\ref{firstc1c2bt2jk}), we obtain  {\allowdisplaybreaks\begin{align*}
\trace(\Sigma_{{\bar T},1} (A_{11} P_{{\bar T}, M, 1}+P_{{\bar T}, M, 1} A_{11}^T))&=\trace(P_{{\bar T}, M, 1}(\Sigma_{{\bar T}, 1} A_{11}+ A_{11}^T 
\Sigma_{{\bar T}, 
1}))\\ 
&=\trace(P_{{\bar T}, M, 1}(G_{{\bar T}, 1}^T G_{{\bar T}, 1}- C_1^T C_1)),
          \end{align*}}
so that{\allowdisplaybreaks\begin{align*}
&-2\trace(C_2 P_{{\bar T},M, 2} C_1^T)\\&=2[\trace(\Sigma_{{\bar T},2} P_{{\bar T}, M, 2} A_{21}^T)-\trace(B_1^T\Sigma_{{\bar T}, 1} B_1)+\trace(C_1 
P_{{\bar T}, M, 
1} C_1^T)]\\&\quad +2[\trace(\Sigma_{{\bar T}, 1}F_{{\bar T}, 1}F_{{\bar T}, r}^T)-\trace(G_{{\bar T}, 1}^TG_{{\bar T}} P_{{\bar T},M})].
          \end{align*}}
Inserting this result into equation (\ref{insertforebbt2fh}) provides{\allowdisplaybreaks\begin{align*}
   \epsilon^2=&\trace(\Sigma_{{\bar T}, 2} (B_2 B_2^T+2 P_{{\bar T}, M, 2} A_{21}^T))\\&+2[\trace(\Sigma_{{\bar T}, 1}F_{{\bar T}, 1}F_{{\bar T}, 
r}^T)-\trace(G_{{\bar T}, 1}^TG_{{\bar T}} 
P_{{\bar T},M})]\\&+\trace(C_1 P_{{\bar T}, r} C_1^T)-\trace(B_1^T \Sigma_{{\bar T},1} B_1).
          \end{align*}}
With the integral representations of $P_{{\bar T}, r}$ and $Q_{{\bar T}, r}$ it holds that 
\begin{align*}
&\trace(C_1 P_{{\bar T}, r} C_1^T)=\int_0^{\bar T} \trace(C_1 \expn^{A_{11}s}B_1B_1^T \expn^{A_{11}^Ts} C_1^T)ds\\&=\int_0^{\bar T} \trace(B_1^T 
\expn^{A_{11}^Ts} 
C_1^TC_1 \expn^{A_{11}s}B_1)ds=\trace(B_1^T Q_{{\bar T}, r} B_1).
\end{align*}
So, we have {\allowdisplaybreaks \begin{align*}
\trace(C_1 P_{{\bar T}, r} C_1^T)-\trace(B_1^T \Sigma_{{\bar T}, 1} B_1)=\trace(B_1 B_1^T(Q_{{\bar T}, r}-\Sigma_{{\bar T}, 1})).
          \end{align*}}
Combining equations (\ref{firstc1c2bt2jk}) and (\ref{redgramobsbt2fzu}), we have
\begin{align}\label{difQRSig}
   A_{11}^T (Q_{{\bar T}, r}-\Sigma_{{\bar T}, 1})+(Q_{{\bar T}, r}-\Sigma_{{\bar T}, 1}) A_{11}  =G_{{\bar T}, r}^T G_{{\bar T}, r}-G_{{\bar T}, 
1}^T 
G_{{\bar T},1}.
\end{align}
Inserting (\ref{nonbalregramp}) and (\ref{difQRSig}) gives {\allowdisplaybreaks\begin{align*}
&\trace(C_1 P_{{\bar T}, r} C_1^T)-\trace(B_1^T \Sigma_{{\bar T}, 1} B_1)\\&=-\trace((A_{11} P_{{\bar T}, r}+P_{{\bar T}, r} A_{11}^T-F_{{\bar T}, r} 
F_{{\bar T}, 
r}^T)(Q_{{\bar T}, 
r}-\Sigma_{{\bar T}, 1}))\\&=-\trace(P_{{\bar T}, r}((Q_{{\bar T}, r}-\Sigma_{{\bar T}, 1})A_{11}+A_{11}^T(Q_{{\bar T}, r}-\Sigma_{{\bar 
T},1})))\\&\quad +\trace(F_{{\bar T}, r} F_{{\bar T}, 
r}^T(Q_{{\bar T}, r}-\Sigma_{{\bar T}, 1}))\\&=\trace(P_{{\bar T}, r}(G_{{\bar T}, 1}^T G_{{\bar T},1}-G_{{\bar T}, r}^T G_{{\bar T}, 
r}))+\trace(F_{{\bar T}, r} F_{{\bar T}, 
r}^T(Q_{{\bar T}, r}-\Sigma_{{\bar T}, 1})).
          \end{align*}}
Using again the integral representations of $P_{{\bar T}, r}$ and $Q_{{\bar T}, r}$, we see that {\allowdisplaybreaks \begin{align*}
          \trace(P_{{\bar T}, r}G_{{\bar T}, r}^T G_{{\bar T},r})&=\int_0^{\bar T} \trace(\expn^{A_{11}s}B_1B_1^T \expn^{A_{11}^Ts} 
\expn^{A^T_{11}{\bar T}}C_1C_1^T\expn^{A_{11}{\bar T}})ds 
\\& =\int_0^{\bar T} \trace(C_1^T\expn^{A_{11}s}\expn^{A_{11}{\bar T}}B_1B_1^T \expn^{A^T_{11}{\bar T}}\expn^{A_{11}^Ts}C_1)ds \\& =\int_0^{\bar T} 
\trace(B_1^T \expn^{A^T_{11}{\bar T}}\expn^{A_{11}^Ts}C_1 C_1^T\expn^{A_{11}s}\expn^{A_{11}{\bar T}}B_1)ds\\&=\trace(F_{{\bar T}, r}^TQ_{{\bar T}, r} 
F_{{\bar T}, 
r})=\trace(F_{{\bar T}, r} F_{{\bar T}, r}^TQ_{{\bar T}, r}).                                                                           \end{align*}}
Hence, we have {\allowdisplaybreaks \begin{align*}
&\trace(C_1 P_{{\bar T}, r} C_1^T)-\trace(B_1^T \Sigma_{{\bar T}, 1} B_1)=\trace(P_{{\bar T}, r}G_{{\bar T}, 1}^T G_{{\bar T},1})-\trace(F_{{\bar T}, 
r} F_{{\bar T}, 
r}^T\Sigma_{{\bar T}, 1}).
\end{align*}}
The error bound $\epsilon^2$ then is{\allowdisplaybreaks\begin{align*}
   \epsilon^2=&\trace(\Sigma_{{\bar T}, 2} (B_2 B_2^T+2 P_{{\bar T}, M, 2} A_{21}^T))\\&+2[\trace(\Sigma_{{\bar T}, 1}F_{{\bar T}, 1}F_{{\bar T}, 
r}^T)-\trace(G_{{\bar T}, 1}^TG_{{\bar T}} 
P_{{\bar T},M})]\\&+\trace(P_{{\bar T}, r}G_{{\bar T}, 1}^T G_{{\bar T},1})-\trace(F_{{\bar T}, r} F_{{\bar T}, r}^T\Sigma_{{\bar T}, 1}).
          \end{align*}}
Since {\allowdisplaybreaks \begin{align*}
   &    2\trace(\Sigma_{{\bar T}, 1}F_{{\bar T}, 1}F_{{\bar T}, r}^T)=2\left\langle \Sigma_{{\bar T}, 1}^{\frac{1}{2}} F_{{\bar T}, r}, \Sigma_{{\bar 
T}, 1}^{\frac{1}{2}}F_{{\bar T}, 
1}\right\rangle_F\\=&\left\|\Sigma_{{\bar T}, 1}^{\frac{1}{2}} F_{{\bar T}, r}\right\|^2_F+\left\|\Sigma_{{\bar T}, 1}^{\frac{1}{2}} F_{{\bar T}, 
1}\right\|^2_F-\left\|\Sigma_{{\bar T}, 1}^{\frac{1}{2}} (F_{{\bar T}, 1}-F_{{\bar T}, r})\right\|^2_F,
      \end{align*}}
 we obtain {\allowdisplaybreaks\begin{align*}
   \epsilon^2=&\trace(\Sigma_{{\bar T}, 2} (B_2 B_2^T+2 P_{{\bar T}, M, 2} A_{21}^T))\\&+\trace(\Sigma_{{\bar T}, 1}F_{{\bar T}, 1}F_{{\bar T}, 
1}^T))-2\trace(P_{{\bar T},M} G_{{\bar T}, 1}^TG_{{\bar T}} 
)+\trace(P_{{\bar T}, r}G_{{\bar T}, 1}^T G_{{\bar T},1})\\&-\trace(\Sigma_{{\bar T}, 1} (F_{{\bar T}, 1}-F_{{\bar T}, r})(F_{{\bar T}, 1}-F_{{\bar 
T}, r})^T)
          \end{align*}}
  which is the claimed result.
\end{proof}
\end{thm}
We now discuss the impact of the remainder term $R_{\bar T}:=-2\trace(G^T_{{\bar T}, 1} G_{\bar T} P_{{\bar T}, M})+\trace(G^T_{{\bar T}, 1} G_{{\bar 
T}, 1} P_{{\bar T}, 
r})+\trace(F_{{\bar T}, 1} F^T_{{\bar T}, 1} 
\Sigma_{{\bar T}, 1})$ of the error bound in Theorem \ref{thmmaimn}. Every summand of $R_{\bar T}$ can be bounded from above as follows: 
\begin{align*}
\trace(G^T_{{\bar T}, 1} G_{\bar T} P_{{\bar T}, M})&\leq \left\| G_{{\bar T}, 1}\right\|_F \left\| G_{{\bar T}}\right\|_F \left\| P_{{\bar T}, 
M}\right\|_F,   \\
\trace(F_{{\bar T}, 1} F^T_{{\bar T}, 1} \Sigma_{{\bar T}, 1})&= \left\|\Sigma_{{\bar T}, 1}^{\frac{1}{2}} F_{{\bar T}, 1}\right\|^2_F\leq \left\| 
F_{{\bar T}, 
1}\right\|^2_F 
\trace(\Sigma_{{\bar T},1}),\\
\trace(G^T_{{\bar T}, 1} G_{{\bar T}, 1} P_{{\bar T}, r})&= \left\|P_{{\bar T}, r}^{\frac{1}{2}} G^T_{{\bar T}, 1}\right\|^2_F\leq \left\| G_{{\bar 
T}, 1}\right\|^2_F 
\trace(P_{{\bar T},r}).
\end{align*}
If $A$ is asymptotically stable, then the norms $\left\| F_{{\bar T}, 1}\right\|_F$,$\left\| G_{{\bar T}, 1}\right\|_F$ and $\left\| G_{{\bar 
T}}\right\|_F$ 
decay exponentially fast, i.e., they are bounded by $c_1\expn^{-c_2{\bar T}}$, where $c_1, c_2>0$ are suitable constants.\smallskip

Now, if the terminal time $\bar T$ is sufficiently large, the term $R_{\bar T}$ is small and hence it can be neglected in the error bound. For very 
stable systems ($c_2$ is large), $\bar T$ can be chosen small and for slowly decaying systems (small constant $c_2$), $\bar T$ needs to be large in 
order to have a sufficiently small $R_{\bar T}$. If the remainder term $R_{\bar T}$ is small, it can be concluded from Theorem \ref{thmmaimn} that 
TLBT works well if the truncated time-limited singular values $\sigma_{{\bar T}, r+1}, \ldots, \sigma_{{\bar T}, n}$ are small.\smallskip

For non-stable systems the remainder term $R_{\bar T}$ in the error bound is expected to be large (exponential growth) which might be an 
indicator for a large error when applying TLBT to these systems.
\begin{bem}
The representation in Theorem \ref{thmmaimn} is not appropriate to determine the error bound since $B_2$ and $A_{21}$ are never computed in practice. 
However, for asymptotically stable systems \eqref{lti} ($R_{\bar T}$ is expected to be small) we know that the reduced order dimension $r$ has to be 
chosen such that $\sigma_{{\bar T}, r+1}, \ldots, \sigma_{{\bar T}, n}$ are small in order to guarantee a good approximation. Consequently, looking 
at the time-limited singular values instead of computing the error bound~\eqref{implicdeptilp} provides an alternative way to find a suitable reduced 
order dimension.
\end{bem}
%%%%%%%%%%%%%%%%%%%%%%%%%%%%%%%%%%
\section{Practical Considerations}
Here we review the practical execution of TLBT for large-scale systems and evaluate the usefulness of the error bound~\eqref{implicdeptilp} in actual
computations. Directly solving the Lyapunov equations~\eqref{BT_Lyap}, \eqref{TLBT_Lyap} is infeasible for large dimensions.
Therefore, for large-scale systems it has become common practice to approximate the Gramians by low-rank factorizations, e.g.,
$P_{\infty}\approx Z_{\infty}Z_{\infty}^T$ with low-rank factors $Z_{\infty}\in\mathbb{R}^{n\times h}$, rank$(Z_{\infty})=h\ll n$, and similarly for 
the other
Gramians. This is justified by the often observed and proven fast singular value decay of solutions of Lyapunov equations~\cite{Gra04}, especially if 
$p,m\ll
n$. For this situation there exist efficient algorithms~\cite{BenS13,Sim16} employing techniques from sparse numerical linear algebra for computing 
the low-rank
solution factors. For the Lyapunov equations~\eqref{TLBT_Lyap} in TLBT, a rational Krylov subspace method~\cite{DruS11} is proposed in~\cite{morKue17}
that is also able to deal with the arising matrix exponentials. With low-rank approximations $P_{\bar T}\approx Z_{P_{\bar T}}Z_{P_{\bar T}}^T$, 
$Q_{\bar
T}\approx
Z_{Q_{\bar T}}Z_{Q_{\bar T}}^T$, one computes the SVD $X\Sigma Y^T=Z_{P_{\bar T}}^TZ_{Q_{\bar T}}$ and projection matrices
$V=Z_{P_{\bar T}} Y_1\Sigma_1^{-\tfrac{1}{2}}$ and $W:=Z_{Q_{\bar T}}X_1\Sigma_1^{-\tfrac{1}{2}}$, where $\Sigma_1$ contains the largest $r$ singular 
values and
$X_1,Y_1$ the associated singular vectors.
The reduced order model $\boldsymbol{\Sigma}_r$ is obtained via $A_{11}:=W^TAV$, $B_1:=W^TB$, $C_1:=CV$ which makes it clear that some of the 
quantities of the
bound in Theorem~\ref{thmmaimn} are not accessible in practical computations.

However, we may nevertheless acquire an approximation of~\eqref{implicdeptilp}. For this $\trace \left(C P_{\bar T}C^T\right)$ can be approximated by 
$\trace
\left(CZ_{P_{\bar T}}^T Z_{P_{\bar T}}C^T\right)$,
$\trace \left(C_1 P_{\bar T,r}C_1^T\right)$ requires solving the $r$ dimensional Lyapunov equation~\eqref{nonbalregramp}, and $\trace \left(C P_{\bar
T,M}C_1^T\right)$
requires the solution of the Sylvester equation~\eqref{mixeqbt2blsa}, which amounts to solve $r$ linear systems of equations defined by $A-\alpha I$,
$\alpha\in\Lambda(A_{11})$ see, e.g.,~\cite[Algorithm~7.6.2]{GolV13}. 
Unlike the error bound in BT~\eqref{BT_errorB}, the TLBT bound~\eqref{implicdeptilp} cannot be easily used to adjust the reduced order because when 
changing $r$
to, say, $r+d$, $d\geq 1$,
the solutions of~\eqref{lyapeq_btbound} have to be computed entirely from scratch. Especially because of the Sylvester equation~\eqref{mixeqbt2blsa}, 
this would
be increasingly expensive.

TLBT can with minor adjustments be applied to generalized state-space systems
\begin{align}\label{glti}
\boldsymbol{\Sigma}:\quad E\dot x(t)=Ax(t)+Bu(t),\quad x(0)=0,\quad y(t)=Cx(t)
\end{align}
with $E$ nonsingular. In that case the time-limited Gramians are $P_{\bar T}$, $E^TQ_{\bar T}E$, where $P_{\bar T}$, $Q_{\bar T}$ solve the
generalized Lyapunov equations
\begin{align}\label{TLBT_GLyap}
\begin{split}
 AP_{\bar T}E^T+EP_{\bar T}A^T+BB^T-F^E_{\bar T}(F^E_{\bar T})^T&=0,\\
A^TQ_{\bar T}E+E^TQ_{\bar T}A^T+C^TC-(G^E_{\bar T})^TG^E_{\bar T}&=0
\end{split}
\end{align}
with $F^E_{t}:=E\expn^{E^{-1}A t}E^{-1}B$ and $G^E_{t}:=C\expn^{E^{-1}At}$, see~\cite{morKue17}. Hence, the derivations of Section~\ref{sec2} can be 
carried out
as before by using the quantities in \eqref{TLBT_GLyap}. In particular, in the constant in the bound~\eqref{implicdeptilp},
$P_{\bar T,M}$ has to be replaced by the solution $P^E_{\bar T,M}$ of 
\begin{align*}
 AP^E_{\bar T,M}+EP^E_{\bar T,M}A_{11}+B\tilde B_1-F^E_{\bar T}(F^E_{{\bar T}, r})^T=0,
\end{align*}
where $SE^{-1}B=\begin{bmatrix}\tilde{B}_1\\ \tilde{B}_2\end{bmatrix}$, $F^E_{{\bar T}, r}:=\expn^{A_{11}\bar T}\tilde B_1$. Here we employed that the 
mass
matrix $E$ is transformed to the identity in (TL)BT. The transformation matrices $V,W$ for TLBT are constructed as before but using the SVD $X\Sigma
Y^T=Z_{P_{\bar T}}^TEZ_{Q_{\bar T}}$, where $Z_{P_{\bar T}},~Z_{Q_{\bar T}}$ are low-rank solution factors of~\eqref{TLBT_GLyap}.
%%%%%%%%%%%%%%%%%%%%%
\section{Numerical Experiments}
All following computations are
carried out in \matlab~8.0.0.783 on a \intel\xeon CPU X5650 (2.67GHz, 48 GB RAM).  
We use the rail model from the Oberwolfach benchmark collection\footnotemark[1] which represents a finite element discretization of a cooling process 
of a steel
rail. \footnotetext[1]{http://portal.uni-freiburg.de/imteksimulation/downloads/benchmark}
It provides symmetric positive and negative definite matrices $M$ and, respectively, $A$, as well as $B\in\mathbb{R}^{n\times 7}$, 
$C\in\mathbb{R}^{6\times n}$.
We begin
with the coarsest discretization level with $n=1357$ which still allows to compute the matrix exponentials and Lyapunov solutions by direct methods.
The final time is $\bar T=100$, the input chosen as $u(t)=50\mathbf{1}_7$ ($\mathbf{1}_h:=[1,\ldots,1]^T\in\mathbb{R}^h$), and the time integration is 
carried
out using an implicit midpoint rule until $T=400$ with a fixed time step $\delta t=0.04$. We generate reduced order models of dimension $r=40$ by both 
BT and
TLBT. Figure~\ref{fig:rail1_error} shows the obtained errors $\|y(t)-y_r(t)\|_2$ and the bound~\eqref{implicdeptilp}, clearly indicating that the 
proposed
bound is valid. Of course, after leaving $[0,\bar T]$, \eqref{implicdeptilp} is no longer valid and 
$\|y(t)-y_r(t)\|_2>\epsilon\left\|u\right\|_{L^2_{\bar T}}$
for some $t>\bar T$.
We also see that ordinary BT provides less accurate reduced order models. It is important to point out that almost identical results were obtained
if low-rank Gramian approximations computed by rational Krylov subspace methods~\cite{DruS11,morKue17} are used. In particular, running the method for 
the
restricted Gramians with the same settings as in~\cite{morKue17} led to $\vert \epsilon^{\text{approx.}}-\epsilon^{\text{exact}}\vert\approx 1.6\cdot 
10^{-9}$
and visually indistinguishable error norms $\|y(t)-y_r(t)\|_2$. 

\begin{figure}[t]
  \centering
  \includegraphics{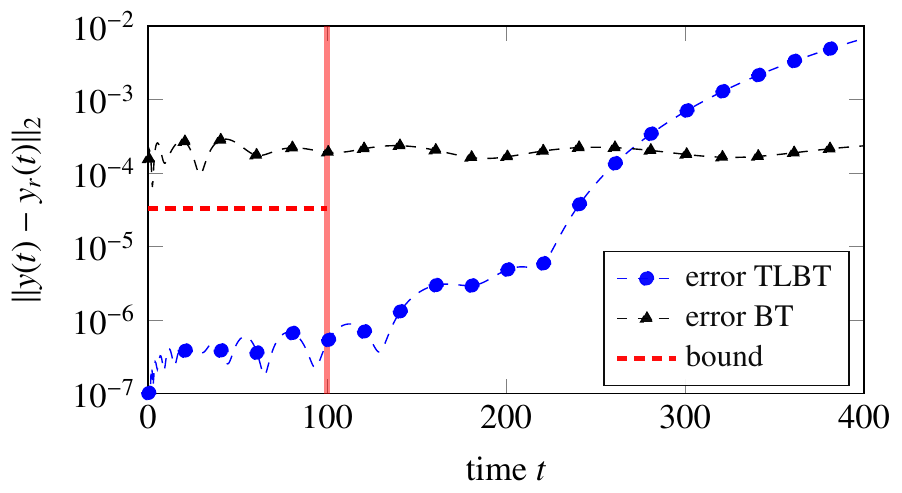}
  \caption{Results obtained by BT and TLBT for small rail model ($n=1357$, $\bar T=100$, $u(t)=50\mathbf{1}_7$, $r=40$).}\label{fig:rail1_error}
\end{figure}

We continue by investigating the influence of the final time $\bar T$ and the reduced order
$r$ to $\max\limits_{t\in [0, {\bar T}]}\|y(t)-y_r(t)\|_2$ and~\eqref{implicdeptilp}. 
The results are visualized in Figure~\ref{fig:rail1_errvs_rT}. For the top plot we fixed $\bar T=100$ and varied the reduced order
$r=10,\ldots,100$. Apparently, TLBT achieves smaller
errors than BT for increasing $r$. After some value of $r$, the bound~\eqref{implicdeptilp} appears to stagnate and fails to capture the decreasing 
behavior of
the error.
The bottom plot shows the results for a fixed  $r=50$ but different final times $\bar T=50,\ldots,300$ which for TLBT requires, naturally, computing
(approximations of) the matrix exponentials and $P_{\bar T},~Q_{\bar T}$ for each value of $\bar T$.  The results indicate that increasing $\bar T$ 
also 
increases the achieved error and the bound~\eqref{implicdeptilp} appears to capture this behavior.
As investigated for TLBT in~\cite{morKue17}, for even larger final times $\bar T$, TLBT will at some point produce errors which are very close to 
those of BT.
% and, thus, virtually rendering the time restrictions redundant.
\begin{figure}[t]
  \centering
\includegraphics{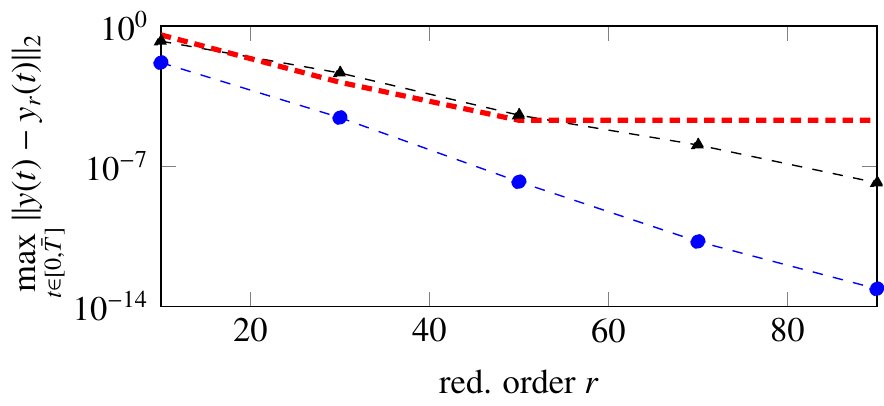} %\input{rail1k_errVSredord}
\\
\quad\includegraphics{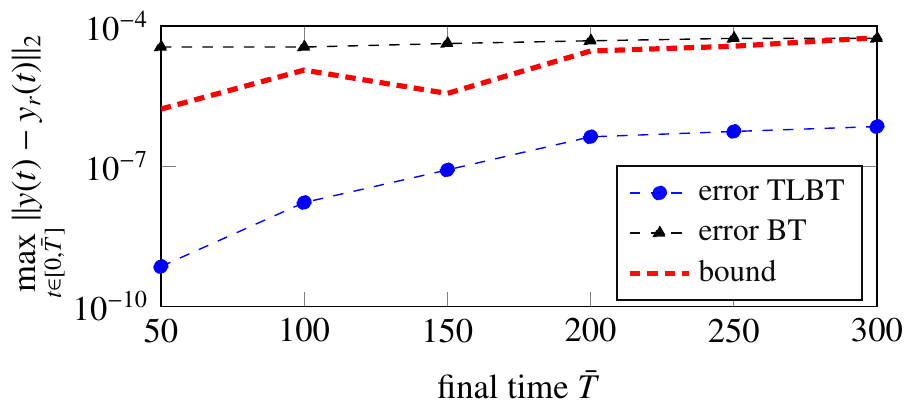}%\input{rail1k_errVSbarT}
  \caption{Influence of $r$ (top) and $\bar T$ (bottom) for small rail model.}
  \label{fig:rail1_errvs_rT}
\end{figure}

Next we experiment with a larger version of the rail model with $n=79841$. This size requires using low-rank solution factors of the Gramians.
We set $u(t)=u_*(t):=[\sin(4t\pi/100),\cos(t\pi/100),3,\expn^{-2t},\cos(t/100)\expn^{-t},\tfrac{1}{1+t^2},\tfrac{1}{1+\sqrt{t}}]^T$ and $\bar T=150$. 
Motivated by Theorem~\ref{thmmaimn}, we experiment with an automatic determination of the reduced order $r$ s.t. $\sum_{i=r+1}^{\hat n}\sigma_{i,\bar
T}\leq \tau$ for some specified tolerance $0<\tau\ll 1$ and $\hat n:=\min(\text{rank}(Z_{P_{\bar T}}),\text{rank}(Z_{Q_{\bar T}}))$, i.e., similar as 
in 
unrestricted BT. The obtained reduced orders $r$ in BT and TLBT, as well as the largest errors in $[0,\bar T]$ and~\eqref{implicdeptilp} are shown in
Figure~\ref{fig:rail79_errvstau} against different values $\tau=10^{-7},\ldots,10^{-2}$.

 TLBT again achieves smaller errors than BT and approximately two orders of magnitude smaller than $\tau$. Note that the obtained reduced orders $r$ 
of TLBT are
for $\tau=10^{-4},10^{-3},10^{-2}$ slightly larger than those of BT.
This experiment nevertheless suggests that choosing the order $r$ in TLBT automatically by looking at the time-limited singular values is as reliable 
as in BT.
\begin{figure}[t]
  \centering
 \includegraphics{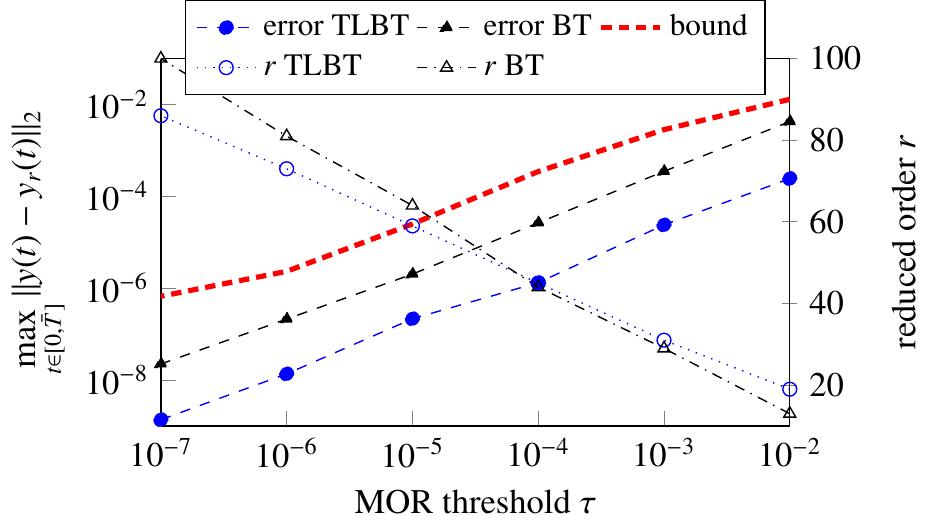}
  \caption{Automatically adjusted orders $r$, maximum errors, bound ~\eqref{implicdeptilp} against tolerances $\tau$ for the larger rail model  
($n=79841$,
$\bar T=150$, $u(t)=u_*(t)$).}
  \label{fig:rail79_errvstau}
\end{figure}
%%%%%%%%%%%%%%%%%%%%%
\section{Conclusion}
In this paper, we have studied time-limited balanced truncation, an alternative to conventional balanced truncation. This scheme can outperform the 
conventional ansatz when seeking for a good reduced order model on a certain finite time interval but, so far, no theory on error bounds has 
been established. Therefore, we proved an $\mathcal H_2$ error bound in this work. We provided two different representations for the bound. One is 
appropriate for practical computations, whereas the other one shows that the time-limited singular values can be used as well in order to 
determine a suitable reduced order dimension. This paper also contains numerical experiments in which we presented the performance of the error bound.

\section*{Acknowledgements}
The authors thank the organizers of the LMS-EPSRC Durham Symposium on Model Order Reduction. The stimulating atmosphere during this meeting has 
resulted in the development of the ideas behind this paper. Moreover, the authors thank Peter Benner for his helpful comments.

\bibliographystyle{plain}

\end{document}